\documentclass[11pt,reqno]{amsart}
 \usepackage{amssymb,amsfonts,amscd,amsbsy}
\usepackage[mathscr]{eucal}
\usepackage{url,color}

\newtheorem{thm}{Theorem}[section]
\newtheorem{lem}[thm]{Lemma}
\newtheorem{prop}[thm]{Proposition}
\newtheorem{cor}[thm]{Corollary}
\theoremstyle{definition}

\newcommand{\la}{\lambda}
\newcommand{\spt}{\operatorname{spt}}
\newcommand{\ospt}{\operatorname{ospt}}
\newcommand{\oospt}{\overline{\operatorname{ospt}}}
\newcommand{\ST}{\operatorname{ST}}
\newcommand{\STe}{\overline{\operatorname{ST}}_{e}}
\newcommand{\STo}{\overline{\operatorname{ST}}_{o}}
\allowdisplaybreaks

\numberwithin{equation}{section}

\makeatletter
\def\imod#1{\allowbreak\mkern5mu({\operator@font mod}\,\,#1)}
\makeatother

\begin{document}

\title[First rank and crank moments for overpartitions]
{The first positive rank and crank moments for overpartitions}

\author{George Andrews}

\author{Song Heng Chan}

\author{Byungchan Kim}

\author{Robert Osburn}

\address{Department of Mathematics, The Pennsylvania State University, University Park, PA 16802, USA}

\address{Division of Mathematical Sciences, School of Physical and Mathematical Sciences, Nanyang Technological University, 21 Nanyang link, Singapore 637371, Republic of Singapore}

\address{School of Liberal Arts \\ Seoul National University of Science and Technology \\ 232 Gongreung-ro, Nowon-gu, Seoul,139--743, Korea}

\address{School of Mathematical Sciences, University College Dublin, Belfield, Dublin 4, Ireland}

\email{andrews@math.psu.edu}

\email{chansh@ntu.edu.sg}

\email{bkim4@seoultech.ac.kr}

\email{robert.osburn@ucd.ie}

\thanks{The first author was partially supported by National Security Agency Grant H98230-12-1-0205. The second author was partially supported by the Singapore Ministry of Education Academic Research Fund, Tier 1,  project number RG68/10.
The third author was supported by Basic Science Research Program through the National Research Foundation of Korea (NRF) funded by the Ministry of Science, ICT \& Future Planning (NRF2011-0009199), and the TJ Park Science Fellowship from the POSCO TJ Park Foundation.}

\date{\today}

\begin{abstract}
In 2003, Atkin and Garvan initiated the study of rank and crank moments for ordinary partitions. These moments satisfy a strict inequality. We prove that a strict inequality also holds for the first rank and crank moments of overpartitions and consider a new combinatorial interpretation
in this setting.
\end{abstract}

\maketitle

\section{Introduction}

A partition of a non-negative integer $n$ is a non-increasing sequence of positive integers whose sum is $n$. For example, the $5$ partitions of $4$ are

\begin{equation*}
\begin{gathered}
4, 3+1, 2+2, 2+1+1, 1+1+1+1.
\end{gathered}
\end{equation*}

\noindent In 1944, Dyson introduced the rank of a partition as the largest part minus the number of parts \cite{dyson}. In 1988, the first author and Garvan defined the crank of a partition as either the largest part, if 1 does not occur as a part, or the difference between the number of parts larger than the number of $1$'s and the number of $1$'s, if $1$ does occur \cite{ag}. These two statistics give a combinatorial explanation of Ramanujan's congruences for the partition function modulo $5$, $7$ and $11$. Let $N(m,n)$ denote the number of partitions of $n$ whose rank is $m$ and $M(m,n)$ the number of partitions of $n$ whose crank is $m$.

A recent development in the theory of partitions has been the study of rank and crank moments as initiated by Atkin and Garvan \cite{atg}. For $k \geq 1$, the $k$th rank moment $N_{k}(n)$ and the $k$th crank moment $M_{k}(n)$ are given by

\begin{equation} \label{krank}
N_{k}(n):= \sum_{m \in \mathbb{Z}} m^{k} N(m,n)
\end{equation}

\noindent and

\begin{equation} \label{kcrank}
M_{k}(n):= \sum_{m \in \mathbb{Z}} m^{k} M(m,n).
\end{equation}

\noindent As $N(-m,n)=N(m,n)$ \cite{dyson} and $M(-m,n)=M(m,n)$ \cite{ag}, we have $N_{k}(n)=M_{k}(n)=0$ for $k$ odd. The even moments are of considerable interest as they have been the subject of a number of works \cite{an1, an2, br, bgm, bm, bmr, bmr2, djr, fo, gar1, gar2, r}. In particular, Garvan \cite{gar1} conjectured that

\begin{equation} \label{rcforpart}
M_{2j}(n) > N_{2j}(n)
\end{equation}

\noindent for all $j$, $n \geq 1$. In \cite{bmr}, (\ref{rcforpart}) was proved for fixed $j$ and sufficiently large $n$. Garvan proved (\ref{rcforpart}) for all $j$ and $n$ via symmetrized rank and crank moments and Bailey pairs \cite{gar2}. Recently, the first three authors gave an elementary proof of (\ref{rcforpart}) by considering modified versions of (\ref{krank}) and (\ref{kcrank}). Namely, consider the positive rank and crank moments

\begin{equation*}
{N}_{k}^{+}(n):= \sum_{m=1}^{\infty} m^{k} N(m,n)
\end{equation*}

\noindent and

\begin{equation*}
{M}_{k}^{+}(n) := \sum_{m=1}^{\infty} m^{k} M(m,n).
\end{equation*}

\noindent In \cite{ack}, it was proved that

\begin{equation} \label{oddrc}
{M}_{k}^{+}(n) > {N}_{k}^{+}(n)
\end{equation}

\noindent for all $k$, $n \geq 1$ by a careful study of the decomposition of the generating function for the difference $M_{k}^{+}(n) - N_{k}^{+}(n)$. For a discussion concerning the asymptotic behavior of these moments, see \cite{bm2}. Inequality (\ref{oddrc}) combined with the fact that $N_{2j}(n)=2N_{2j}^{+}(n)$ and $M_{2j}(n)=2M_{2j}^{+}(n)$ imply (\ref{rcforpart}).

Our interest in this paper is to consider an analogue of (\ref{oddrc}) for overpartitions. More specifically, we will investigate the first moments for overpartitions and what is counted by the difference. Recall that an overpartition \cite{lc} is a partition in which the first occurrence of each distinct number may be overlined. For example,  the $14$ overpartitions of $4$ are

\begin{equation*}
\begin{gathered}
4, \overline{4}, 3+1, \overline{3} + 1, 3 + \overline{1},
\overline{3} + \overline{1}, 2+2, \overline{2}
+ 2, 2+1+1, \overline{2} + 1 + 1, 2+ \overline{1} + 1, \\
\overline{2} + \overline{1} + 1, 1+1+1+1, \overline{1} + 1 + 1 +1.
\end{gathered}
\end{equation*}

\noindent These combinatorial objects have recently played an important role in the construction of weight 3/2 mock modular forms \cite{bl}, in Rogers-Ramanujan and Gordon type identities \cite{css} and in the study of Jack superpolynomials in supersymmetry and quantum mechanics \cite{dlm}.

Let $\overline{N}(n,m)$ denote the number of overpartitions of $n$ whose rank is $m$ and $\overline{M}(n,m)$ the number of overpartitions of $n$ whose (first residual) crank is $m$. Here, Dyson's rank extends easily to overpartitions and the first residual crank of an overpartition is obtained by taking the crank of the subpartition consisting of the non-overlined parts \cite{blo}. It is now natural to consider the rank and crank overpartition moments

\begin{equation*}
\overline{N}_{k}(n) :=\sum_{m \in \mathbb{Z}} m^{k} \overline{N}(m,n)
\end{equation*}

\noindent and

\begin{equation*}
\overline{M}_{k}(n) := \sum_{m \in \mathbb{Z}} m^{k} \overline{M}(m,n).
\end{equation*}

\noindent Via the symmetries $\overline{N}(-m,n)=\overline{N}(m,n)$ \cite{love} and $\overline{M}(-m,n)=\overline{M}(m,n)$ \cite{blo}, we have $\overline{N}_{k}(n)=\overline{M}_{k}(n)=0$ for $k$ odd. Thus, to obtain non-trivial odd moments, we consider

\begin{equation*}
\overline{N}_{k}^{+}(n) := \sum_{m=1}^{\infty} m^{k} \overline{N}(m,n)
\end{equation*}

\noindent and

\begin{equation*}
\overline{M}_{k}^{+}(n) := \sum_{m=1}^{\infty} m^{k} \overline{M}(m,n).
\end{equation*}

\noindent The main result in this paper is an analogue of (\ref{oddrc}) for overpartitions in the case $k=1$.

\begin{thm} \label{main} For all $n \geq 1$, we have
\begin{equation} \label{oddrcover}
\overline{M}_{1}^{+}(n) > \overline{N}_{1}^{+}(n).
\end{equation}
\end{thm}

%
%
%

The paper is organized as follows. In Section 2, we prove Theorem \ref{main}. In Section 3, we give a combinatorial interpretation of $\overline{M}_{1}^{+}(n) - \overline{N}_{1}^{+}(n)$. In Section 4, we conclude with some remarks regarding future directions.

\section{The proof of Theorem \ref{main}}

For $k\geq 1$, we define the generating functions

\begin{equation*}
\overline{M}_{k}(q)=\sum_{n=1}^{\infty} \overline{M}_{k}^{+}(n) q^n
\end{equation*}

\noindent and

\begin{equation*}
\overline{R}_{k}(q)=\sum_{n=1}^{\infty} \overline{N}_{k}^{+}(n) q^n
\end{equation*}

\noindent and compute their explicit expressions for $k=1$. Throughout, we use the standard $q$-hypergeometric notation,

\begin{equation*}
(a)_n = (a;q)_n = \prod_{k=1}^{n} (1-aq^{k-1}),
\end{equation*}

\noindent valid for $n \in \mathbb{N} \cup \{\infty\}$. For convenience, we define $(a ;q )_0 =1$.

\begin{prop} \label{r1m1gen} We have

\begin{equation} \label{r1}
\overline{R}_{1}(q) = \frac{2(-q)_{\infty}}{(q)_{\infty}} \sum_{n=1}^{\infty} (-1)^{n+1} \frac{q^{n(n+1)}}{1-q^{2n}}
\end{equation}

\noindent and

\begin{equation} \label{m1}
\overline{M}_{1}(q) = \frac{(-q)_{\infty}}{(q)_{\infty}} \sum_{n=1}^{\infty} (-1)^{n+1} \frac{q^{n(n+1)/2}}{1-q^n}.
\end{equation}

\end{prop}

\begin{proof}
We begin with the generalized Lambert series representation of the two-variable generating function for Dyson's rank for overpartitions,

\begin{equation} \label{rzq}
\begin{aligned}
\overline{R}(z,q) & := \sum_{n=0}^{\infty} \sum_{m \in \mathbb{Z}} \overline{N}(m,n) z^m q^n \\
& = \sum_{n=0}^{\infty} \frac{(-1)_{n} q^{n(n+1)/2}}{(zq)_{n} (q/z)_{n}} \\
& = \frac{(-q)_{\infty}}{(q)_{\infty}} \Biggl ( 1 + 2 \sum_{n=1}^{\infty} \frac{(1-z) (1-1/z) (-1)^n q^{n^2 + n}}{(1-zq^n)(1-q^n/z)} \Biggr) \\
& = \frac{(-q)_{\infty}}{(q)_{\infty}} \Biggl ( 1 + 2\sum_{n=1}^{\infty} (-1)^n q^{n^2} - 2  \sum_{n=1}^{\infty} \frac{(-1)^n q^{n^2} (1-q^n)}{1+q^n} \Biggl( \sum_{m=0}^{\infty} z^m q^{mn} + \sum_{m=1}^{\infty} z^{-m} q^{mn} \Biggr) \Biggr).
\end{aligned}
\end{equation}

\noindent For the second and third equalities in (\ref{rzq}), see the proof of Proposition 3.2 in \cite{love}. Here, we have used the identity

\begin{equation*}
\frac{(1-z)(1-1/z) q^{n}}{(1-zq^n)(1-q^n/z)} = 1 - \frac{1-q^n}{1+q^n} \Biggl( \sum_{m=0}^{\infty} z^m q^{mn} + \sum_{m=1}^{\infty} z^{-m} q^{mn} \Biggr)
\end{equation*}

\noindent for the last equality in (\ref{rzq}). We now apply the differential operator $z\frac{\partial}{\partial z}$ to both sides of (\ref{rzq}) to obtain

\begin{equation} \label{operated}
\begin{aligned}
z & \frac{\partial}{\partial z} \Bigl( \overline{R}(z,q) \Bigr) \\
& = \frac{(-q)_{\infty}}{(q)_{\infty}} \Biggl( 2 \sum_{n=1}^{\infty} \frac{(-1)^{n+1} q^{n^2} (1-q^n)}{1+q^n} \sum_{m=1}^{\infty} m z^m q^{mn} + 2 \sum_{n=1}^{\infty} \frac{(-1)^n q^{n^2} (1-q^n)}{1+q^n} \sum_{m=1}^{\infty} m z^{-m} q^{mn} \Biggr).
\end{aligned}
\end{equation}

\noindent Only the first term on the right side of (\ref{operated}) contributes to positive powers of $z$ and so

\begin{equation} \label{first}
\begin{aligned}
\overline{R}_{1}(q) & = \lim_{z \to 1} \frac{2(-q)_{\infty}}{(q)_{\infty}} \sum_{n=1}^{\infty} \frac{(-1)^{n+1} q^{n^2} (1-q^n)}{1+q^n} \sum_{m=1}^{\infty} m z^m q^{mn} \\
& = \frac{2(-q)_{\infty}}{(q)_{\infty}} \sum_{n=1}^{\infty} \frac{(-1)^{n+1} q^{n^2} (1-q^n)}{1+q^n} \sum_{m=1}^{\infty} m q^{mn} \\
& = \frac{2(-q)_{\infty}}{(q)_{\infty}} \sum_{n=1}^{\infty} \frac{(-1)^{n+1} q^{n(n+1)}}{1-q^{2n}},
\end{aligned}
\end{equation}

\noindent which is (\ref{r1}). In the last equality of (\ref{first}), we applied the identity

\begin{equation*}
\sum_{m=1}^{\infty} m q^{mn} = \frac{q^{n}}{(1-q^n)^2}.
\end{equation*}

\noindent For the two-variable generating function for the first residual crank for overpartitions \cite{blo}, we have

\begin{equation}
\overline{C}(z,q):= \sum_{n=0}^{\infty} \sum_{m \in \mathbb{Z}} \overline{M}(m,n) z^m q^n = (-q)_{\infty} C(z,q)
\end{equation}

\noindent where $C(z,q)$ is the two-variable generating function for the crank for partitions. Thus, by the proof of Theorem 1 in \cite{ack}, we obtain (\ref{m1}).

\end{proof}

We now require the following two lemmas for the proof of Theorem \ref{main}.

\begin{lem}\label{la} If

\[
h(q) := \sum_{n=1}^{\infty} \frac{(-1)^{n+1} q^{n(n+1)/2}}{1-q^n},
\]

\noindent then

\[
h(q)=\sum_{j=1}^{\infty} q^{j^2} (1+2q^{j} + 2 q^{2j}
+ \cdots + 2 q^{j^2 -j} + q^{j^2} ).
\]
\end{lem}

\begin{proof}
We first note that 
\begin{equation} \label{laeq1}
\sum_{j=1}^\infty \frac{q^j}{1-q^j} = \sum_{j=1}^\infty\sum_{k=1}^\infty q^{jk}
= \sum_{j=1}^\infty \sum_{k\geq j}^\infty q^{jk} + \sum_{k=1}^\infty \sum_{j>k}^\infty q^{jk}
= \sum_{j=1}^\infty \frac{q^{j^2}(1+q^j)}{1-q^j}.
\end{equation}
By employing a similar argument, we can also derive that
\begin{equation}\label{laeq2}
\sum_{j=1}^{\infty} \frac{q^{2j-1}}{1-q^{2j-1}} = \sum_{j=1}^{\infty} \frac{q^{j(j+1)/2}}{1-q^j}.
\end{equation}

By expanding the summation according to the parity of $n$, we find that
\begin{align*}
h(q) 
&=\sum_{n=1}^{\infty} \frac{q^{n(n+1)/2}}{1-q^{n}} -2 \sum_{n=1}^{\infty} \frac{q^{n(2n+1)}}{1-q^{2n}} \\
&=\sum_{n=1}^{\infty} \frac{q^{n(n+1)/2}}{1-q^{n}} +  \sum_{n=1}^{\infty} \frac{q^{2n^2}(1+q^{2n})}{1-q^{2n}} -  \left(  \sum_{n=1}^{\infty} \frac{q^{2n^2}(1+q^{2n})}{1-q^{2n}} +2\sum_{n=1}^{\infty} \frac{q^{n(2n+1)}}{1-q^{2n}} \right)\\
&=\sum_{n=1}^{\infty} \frac{q^{2n-1}}{1-q^{2n-1}} + \sum_{n=1}^{\infty} \frac{q^{2n}}{1-q^{2n}} - \sum_{n=1}^{\infty} \frac{q^{2n^2}(1+q^n)^2}{1-q^{2n}} \quad \text{by \eqref{laeq1} and \eqref{laeq2}} \\
&=\sum_{n=1}^{\infty} \frac{q^n}{1-q^n} - \sum_{n=1}^{\infty} \frac{q^{2n^2}(1+q^n)}{1-q^{n}} \\
&=\sum_{j=1}^\infty \frac{q^{j^2}(1+q^j)}{1-q^j}-\sum_{j=1}^\infty \frac{q^{2j^2} (1+q^{j})}{(1-q^j)} \quad \text{by \eqref{laeq1}} \\
&=\sum_{j=1}^\infty \frac{q^{j^2}(1+q^j)(1-q^{j^2})}{1-q^j}\\
&=\sum_{j=1}^\infty q^{j^2}(1+q^j)(1+q^j+\cdots + q^{j(j-1)})\\
&=\sum_{j=1}^\infty q^{j^2}(1+2q^j+2q^{2j}+\cdots + 2q^{j^2-j}+q^{j^2}).
\end{align*}

\end{proof}

\begin{lem}\label{l1}
\begin{equation} \label{e1}
h(q)-2h(q^2)
=\sum_{n=1}^\infty (-1)^{n+1}q^{n^2} \left( 1-2q^n + 2q^{2n} - \cdots + (-1)^{n-1} 2q^{n^2-n} + (-1)^nq^{n^2} \right).
\end{equation}
\end{lem}

\begin{proof}
Expanding the right side of \eqref{e1} according to the parity of
$n$ and then separating the positive terms from the negative terms, we find that
\begin{equation} \label{right}
\begin{aligned}
&\sum_{n=1}^\infty (-1)^{n+1}q^{n^2} \left( 1-2q^n + 2q^{2n}+ \cdots + (-1)^{n-1} 2q^{n^2-n} + (-1)^nq^{n^2} \right)\\
&=\sum_{n=1}^\infty q^{(2n-1)^2} \left( 1 + 2q^{4n-2}+2q^{8n-4}+ \cdots +  2q^{4n^2-6n+2} \right)\\
&\quad-\sum_{n=1}^\infty q^{(2n-1)^2}\left( 2q^{2n-1} + 2q^{6n-3}+ \cdots +  2q^{4n^2-8n+3} + q^{(2n-1)^2} \right) \\
&\quad+\sum_{n=1}^\infty q^{(2n)^2} \left( 2q^{2n} + 2q^{6n}+ \cdots + 2q^{4n^2-2n}  \right)\\
&\quad-\sum_{n=1}^\infty q^{(2n)^2} \left( 1+ 2q^{4n}+ \cdots + 2q^{4n^2-4n} + q^{(2n)^2} \right).
\end{aligned}
\end{equation}

Using Lemma \ref{la}, we compute a similar expansion for $h(q)$, then compare with (\ref{right}) in order to see that it suffice to prove

\begin{align} \nonumber
h(q^2)
&=\sum_{n=1}^\infty q^{2n^2} \left( 1+2q^{2n} + 2q^{4n}+ \cdots + 2q^{2n^2-2n} + q^{2n^2} \right)\\ \nonumber
&=\sum_{n=1}^\infty q^{(2n-1)^2}\left( 2q^{2n-1} + 2q^{6n-3}+ \cdots +  2q^{4n^2-8n+3} + q^{(2n-1)^2} \right) \\
&\quad+\sum_{n=1}^\infty q^{(2n)^2} \left( 1+ 2q^{4n}+ \cdots + 2q^{4n^2-4n} + q^{(2n)^2} \right). \label{e2}
\end{align}
Subtracting $\displaystyle \sum_{n=1}^\infty q^{2n^2} \left( 1 + q^{2n^2}
\right)$ from both sides of \eqref{e2} and then dividing by 2, it
remains to show that
\begin{align} \nonumber
&\sum_{n=2}^\infty q^{2n^2} \left( q^{2n} + q^{4n}+ \cdots + q^{2n^2-2n} \right)\\
&=\sum_{n=2}^\infty q^{(2n-1)^2}\left( q^{2n-1} + q^{6n-3}+ \cdots +  q^{4n^2-8n+3} \right) +\sum_{n=2}^\infty q^{4n^2} \left( q^{4n}+ \cdots + q^{4n^2-4n} \right).
\label{e3}
\end{align}

Define $f(n,j)=q^{2n^2+2nj}$. Substituting $f(n, j)$ into the left
side of \eqref{e3} and making a change of summation index $k=n+j$,
we find that
\begin{align*}
&\sum_{n=2}^\infty q^{2n^2} \left( q^{2n} + q^{4n}+ \cdots + q^{2n^2-2n} \right)
=\sum_{n=2}^\infty \sum_{j=1}^{n-1} f(n,j) =\sum_{n=2}^\infty\sum_{k=n+1}^{2n-1} f(n, k-n)\\
&=\sum_{l=2}^\infty \sum_{n=l}^{2l-2} f(n,2l-1-n) +\sum_{m=2}^\infty\sum_{n=m+1}^{2m-1} f(n, 2m-n)\\
&=\sum_{l=2}^\infty \left(q^{4l^2-2l} + q^{4l^2+2l-2} + \cdots +
q^{8l^2-12l + 4} \right)
+\sum_{l=2}^\infty \left( q^{4l^2+4l} + q^{4l^2+8l} + \cdots + q^{8l^2-4l} \right),
\end{align*}
where in the penultimate equality, we rearranged the order of
summation and separated the terms into odd and even values of $k$
via $k=2l-1$ and $k=2m$. We see that these are equal to the right
side of \eqref{e3} and this completes the proof.

\end{proof}

We can now prove Theorem \ref{main}

\begin{proof}[Proof of Theorem \ref{main}.]

By Proposition \ref{r1m1gen}, we have

\begin{equation} \label{rmh}
\overline{M}_{1}(q) - \overline{R}_{1}(q) = \frac{(-q)_{\infty}}{(q)_{\infty}} \Bigl ( h(q) - 2 h(q^2) \Bigr).
\end{equation}

\noindent Thus, it suffices to prove that the right side of (\ref{rmh}) has positive power series coefficients for all positive powers of $q$. By Lemma \ref{l1},
\begin{align*} \nonumber
h(q)-2h(q^2) & = \sum_{n=1}^\infty (-1)^{n+1}q^{n^2}
 +
 2\sum_{n=2}^\infty (-1)^{n}q^{n^2} \left(q^n - q^{2n}+ \cdots + (-1)^{n-2} q^{n^2-n}\right)\\
&\quad -2\sum_{n=1}^\infty q^{2(2n)^2}
 -\sum_{n=1}^\infty (-1)^{n+1}q^{2n^2}\\ \nonumber
&=:A_1 + 2A_2 - 2A_3- A_4.
\end{align*}

For the sum $A_1$, note that
\[
-\frac{1}{2}+ A_1 = -\frac{1}{2} \sum_{n=-\infty}^\infty (-1)^nq^{n^2} = -\frac{(q)_\infty}{2(-q)_\infty}.
\]
Hence
\[
\frac{(-q)_\infty}{(q)_\infty} A_1 = \frac{(-q)_\infty}{2(q)_\infty}-\frac{1}{2}.
\]
Similarly, for the sum $A_4$,
\[
\frac{(-q^2;q^2)_\infty}{(q^2;q^2)_\infty} A_4 = \frac{(-q^2;q^2)_\infty}{2(q^2;q^2)_\infty}-\frac{1}{2}.
\]
Therefore,
\begin{align*}
\frac{(-q)_\infty}{(q)_\infty}(A_1-A_4) = \frac{(-q)_\infty}{2(q)_\infty}-\frac{1}{2}
-\frac{(-q;q^2)_\infty}{(q;q^2)_\infty}\left(\frac{(-q^2;q^2)_\infty}{2(q^2;q^2)_\infty}-\frac{1}{2}\right)
=\frac{(-q;q^2)_\infty}{2(q;q^2)_\infty}-\frac{1}{2},
\end{align*}
which has positive power series coefficients for all positive powers of $q$. Next, we examine $A_2-A_3$. We define $g(n,j) = (-1)^{n+j-1}q^{n^2+jn}$. Then
\begin{align*}
A_2 - A_3 &=\sum_{n=2}^\infty (-1)^{n}q^{n^2}\sum_{j=1}^{n-1}(-1)^{j-1} q^{jn} -\sum_{n=1}^\infty q^{2(2n)^2}\\
&= \sum_{n=2}^\infty\sum_{j=1}^{n-1} g(n,j) +\sum_{n=1}^\infty g(2n,2n).\\
\end{align*}

\noindent We now rearrange the series $A_2-A_3$ into several sums. Note that for $j \geq 0$ and $n\geq 2j+2$,
\begin{align*}
&g(2n,4j+3) + g(2n+1,4j+3) + g(2n+1,4j+1) + g(2n+2,4j+1) \\
&=  (-1)^{2n+4j+2}q^{4n^2+(4j+3)2n}\left(1-q^{4n+4j+4}-q^{4j+2}+q^{4n+8j+6}\right)\\
&=  q^{4n^2+(4j+3)2n} \left(1-q^{4j+2}\right)  \left(1-q^{4n+4j+4}\right),
\end{align*}
and for $j\geq 0$ and $n \geq 2j+2$,
\begin{align*}
&g(2n+1,4j+4) + g(2n+2,4j+4) + g(2n+2,4j+2) + g(2n+3,4j+2) \\
&=  (-1)^{2n+4j+4}q^{(2n+1)^2+(4j+4)(2n+1)}\left(1-q^{4n+4j+7}-q^{4j+3}+q^{4n+8j+10}\right)\\
&=  q^{(2n+1)^2+(4j+4)(2n+1)}\left(1-q^{4j+3}\right) \left(1-q^{4n+4j+7}\right).
\end{align*}
These take care of all the terms except, for all integers $n\geq 0$,
\begin{align*}
&g(4n+2,4n+1) + g(4n+3,4n+1) + g(4n+4,4n+1) + g(4n+2,4n+2)\\
&+
g(4n+3,4n+2) + g(4n+4,4n+2) + g(4n+5,4n+2) + g(4n+4,4n+4)\\
&=
[g(4n+2,4n+1) + g(4n+3,4n+1) + g(4n+4,4n+1) + g(4n+2,4n+2) \\
&\quad+ g(4n+3,4n+2)  - g(4n+3,4n+4) ]\\
&\quad+\left[g(4n+3, 4n+4)+ g(4n+4, 4n+4)+g(4n+4, 4n+2)+g(4n+5, 4n+2)\right].
\end{align*}
Note that
\begin{align*}
& g(4n+2,4n+1) + g(4n+3,4n+1) + g(4n+4,4n+1) + g(4n+2,4n+2) \\
& +g(4n+3,4n+2)  - g(4n+3,4n+4) \\
&=q^{(4n+2)^2+(4n+1)(4n+2)}
\left[1-q^{12n+6}+q^{24n+14}-q^{4n+2}+q^{16n+9}-q^{24n+15}\right]\\
&=q^{(4n+2)^2+(4n+1)(4n+2)}
\Big[(1-q^{4n+2})(1-q^{8n+7})(1-q^{12n+6})\\
&\quad +q^{12n+8}(1-q)(1-q^{4n})
+q^{8n+7}(1-q^{4n+1})(1-q^{12n+6})\Big]
\end{align*}
while
\begin{align*}
&g(4n+3, 4n+4)+ g(4n+4, 4n+4)+g(4n+4, 4n+2)+g(4n+5, 4n+2)\\
&=q^{(4n+2)^2+(4n+1)(4n+2)+24n+15}
(1-q^{4n+3})(1-q^{12n+11}).
\end{align*}
These sums show that
\begin{align*}
&\frac{(-q)_\infty}{(q)_\infty}(A_2-A_3)
\\
&=\frac{(-q)_\infty}{(q)_\infty}
\sum_{j= 0}^\infty \sum_{n=2j+2}^\infty q^{4n^2+(4j+3)2n}
\left(1-q^{4j+2}\right) \left(1-q^{4n+4j+4}\right)
\\
&\quad +\frac{(-q)_\infty}{(q)_\infty}
\sum_{j= 0}^\infty \sum_{n=2j+2}^\infty
q^{(2n+1)^2+(4j+4)(2n+1)}\left(1-q^{4j+3}\right) \left(1-q^{4n+4j+7}\right)\\
&\quad +\frac{(-q)_\infty}{(q)_\infty}
\sum_{n=0}^\infty
q^{(4n+2)^2+(4n+1)(4n+2)}
\Big[(1-q^{4n+2})(1-q^{8n+7})(1-q^{12n+6})\\
&\quad +q^{12n+8}(1-q)(1-q^{4n})
+q^{8n+7}(1-q^{4n+1})(1-q^{12n+6})\Big]\\
&\quad +\frac{(-q)_\infty}{(q)_\infty}
\sum_{n=0}^\infty
q^{(4n+2)^2+(4n+1)(4n+2)+24n+15}
(1-q^{4n+3})(1-q^{12n+11}).
\end{align*}

For positive integers $a$, $b$, $c$ and $d$ with $b < c < d$, expressions of the form
\[
\frac{(-q)_\infty}{(q)_\infty}q^a(1-q^b)(1-q^c)
\]
and
\[
\frac{(-q)_\infty}{(q)_\infty}q^a(1-q^b)(1-q^c)(1-q^d)
\]
have nonnegative coefficients and so
$\frac{(-q)_\infty}{(q)_\infty}(A_2-A_3)$ has nonnegative power series coefficients. Since $\frac{(-q)_\infty}{(q)_\infty}(A_1-A_4)$ has positive power series coefficients for all
positive powers of $q$, we conclude that the power series
expansion of $\frac{(-q)_\infty}{(q)_\infty}(h(q)-2h(q^2))$ has
positive coefficients for all $q^n$, $n\geq 1$. This proves (\ref{oddrcover}).
\end{proof}

\begin{cor} \label{2case}
\[
\frac{1}{(q)_\infty}(h(q)-2h(q^2))
\]
has positive power series coefficients for all $q^n$ with $n\geq 6$.
\end{cor}

\begin{proof}
From the proof of Theorem \ref{main} and by invoking the elementary identity $(-q)_\infty = 1/(q;q^2)_\infty$, we see that
\[
\frac{1}{(q)_\infty}(A_1-A_4) = \frac{1}{(-q)_\infty} \left(\frac{(-q;q^2)_\infty}{2(-q;q^2)_\infty}-\frac{1}{2}\right)
=\frac{1}{2}((-q;q^2)_\infty - (q;q^2)_\infty),
\]
which has positive power series coefficients for all odd positive powers of $q$ (the terms with even powers of $q$ vanishes). Again, from the proof of Theorem 1.1, it is easy to see that $\frac{1}{(q)_\infty}(A_2-A_3)$ has nonnegative power series coefficients.
Since one of the terms in the corresponding expression of $\frac{1}{(q)_\infty}(A_2-A_3)$ is
\[
\frac{1}{(q)_\infty}q^6(1-q^2)(1-q^6)(1-q^8) =q^6\prod_{\substack{k=1\\k\neq 2, 6, 8}^\infty}^\infty \frac{1}{1-q^k},
\]
the coefficients of $q^n$ for $n\geq 6$ in the power series expansion of
$\frac{1}{(q)_\infty}(A_2-A_3)$ are all positive.

\end{proof}


\section{A combinatorial interpretation}

In \cite{ack}, the first three authors defined a new counting function $\ospt(n)$ as

\begin{equation*}
\label{osptdefn}
\ospt(n) = M_{1}^{+}(n) - N^{+}_{1}(n)
\end{equation*}

\noindent and provided its combinatorial interpretation. The function $\ospt(n)$ is an interesting companion of $\spt(n)$ in sense of that

\begin{equation*}
\label{sptmoment}
\spt(n) = M_{2}^{+}(n) - N_{2}^{+}(n).
\end{equation*}

\noindent Here, $\spt(n)$ is the number of smallest parts in the partitions of $n$ \cite{an2}. In this section, we discuss an overpartition analogue of $\ospt(n)$ and its combinatorial meaning. Let us define

\[
\oospt(n) = \overline{M}_{1}^{+}(n) - \overline{N}_{1}^{+}(n).
\]

Before giving a combinatorial interpretation for $\oospt(n)$, we first recall the description of $\ospt(n)$. An even string in the partition $\la$ is a sequence of the consecutive parts starting from some even number $2k+2$ where the length is an odd number greater than or equal to $2k+1$ and $2k+2$ plus the length of the string (the number of consecutive parts) do not appear as a part. An odd string in $\la$ is a sequence of the consecutive parts starting from some odd number $2k+1$ where the length is greater than or equal to $2k+1$ such that the part $2k+1$ appears exactly once and $2k+2$ plus the length of the string does not appear as a part. By ``consecutive parts", we allow repeated parts. With these notions in mind, we have the following.

%

\begin{thm}\cite[Theorem 4]{ack} \label{osptcount}
For all positive integers $n$,
\[
\ospt (n) = \sum_{\la \vdash n} \ST (\la),
\]
where the sum runs over the partitions of $n$ and $\ST (\la)$ is the number of even and odd strings in the partition $\la$.
\end{thm}

The function $\oospt(n)$ now counts the number of certain strings in the overpartitions of $n$, but the difference is that we have a weighted count of strings. We start by defining $f_{k} (q)$ as
\[
f_{k} (q) = \sum_{n=1}^{\infty} (-1)^{n+1} q^{n(n+1)/2 + n (k-1)}.
\]
By Proposition \ref{r1m1gen} and exchanging the order of summation, we have
\[
\sum_{n=1}^{\infty} ( \overline{M}_{1}^{+}(n) - \overline{N}_{1}^{+}(n) ) q^n = \frac{(-q)_{\infty}}{(q)_{\infty}} \sum_{k=1}^{\infty} ( f_{k} (q) - 2 f_{k} (q^2) ).
\]
Note that for a fixed $k\geq 1$,
\begin{align*}
&\frac{(-q)_{\infty}}{(q)_{\infty}} (f_{2k-1} (q) + f_{2k} (q) - 2 f_{k} (q^2)) \\
&=\frac{(-q)_{\infty}}{(q)_{\infty}} \sum_{n=1}^{\infty} q^{2n^2-5n+4nk-2k+2}(1-q^{2n^2-n})(1-q^{4n+2k-2}) - q^{2n^2-3n+4nk} (1-q^{2n^2 +n}) (1-q^{4n+2k}).
\end{align*}
 Now we define $A_k(n)$ (resp. $B_k (n)$) to be the number of overpartitions of $n$ counted by the first (resp. second) sum.  By noting that
\[
2n^2 - 3n + 4nk  = 1 + (2k-2) + 2 + (2k-2) + \cdots + (2n-1) + (2k-2) + 2n + (2k-2),
\]
we define an odd string starting from $2k-1$ in an overpartition as
\begin{enumerate}
\item $2k-1, 2k, \ldots, 2\ell+2k-3$ appears at least once, i.e. there are $2\ell -1$ consecutive parts starting from $2k-1$.
\item There is no other part of size $2\ell^2 - \ell$ and $4\ell + 2k -2$.
\end{enumerate}
Similarly, we define an even string starting from $2k$ in an overpartition as
\begin{enumerate}
\item $2k-1, 2k, \ldots, 2\ell+2k-2$ appears at least once, i.e. there are $2\ell$ consecutive parts starting from $2k-1$.
\item There is no other part of size $2\ell^2 + \ell$ and $4\ell + 2k$.
\end{enumerate}

As with the $\ospt(n)$ function, $A_{k} (n)$ is now the number of odd strings starting from $2k-1$ along the overpartitions of $n$, and  $B_{k} (n)$ is the number of even strings starting from $2k-1$ along the overpartitions of $n$. Then we have
\[
\sum_{n=1}^{\infty} ( \overline{M}_{1}^{+}(n) - \overline{N}_{1}^{+}(n) ) q^n = \sum_{n=1}^{\infty}  \sum_{k=1}^{\lfloor(n+1)/2\rfloor} ( A_{k}(n) - B_{k} (n) ) q^n = \sum_{n=1}^{\infty} \oospt(n) q^n.
\]

We have thus proven the following.

\begin{thm} \label{combo}
For all positive integers $n$, we have
\[
\oospt(n) = \STo (n) - \STe (n),
\]
where $\STo(n)$ (resp. $\STe(n)$) is the number of odd (resp. even) strings along the overpartitions of $n$.
\end{thm}

Let us illustrate the above discussion for $n=5$. From Table 1, we see that $\STo(5)=8$ and $\STe(5) = 4$, so $\oospt(5) =4$. This matches with $\overline{M}_{1}^{+}(5) = 24$ and $\overline{N}_{1}^{+}(5) = 20$.

\begin{table}[t]
\begin{center}
\begin{tabular}{c|c|c}
Overpartitions of $5$ & The number of odd strings & The number of even strings  \\ \hline
5& 1 & 0 \\
$\overline{4}$+1 & 1 & 0 \\
3+2 & 1 & 0 \\
3 +$\overline{2}$ & 1 & 0  \\
$\overline{3}$+$\overline{1}$ +1 & 1 & 0\\
3+$\overline{1}$+1 & 1 & 0 \\
2+2+1 & 1 & 1 \\
$\overline{2}$+2+1 & 1 & 1 \\
2+1+1+1 & 0 & 1 \\
2+$\overline{1}$+1+1 & 0 & 1 \\
\end{tabular}
\end{center}
\label{999}
\caption{The number of strings in the overpartitions of $5$.}
\end{table}


\section{Concluding Remarks}

We have numerically observed that

\begin{equation} \label{m2n2}
\overline{M}_{k}^{+}(n) > \overline{N}_{k}^{+}(n)
\end{equation}

\noindent for all $k$, $n \geq 1$. Inequality (\ref{m2n2}) and the fact that $\overline{N}_{2j}(n)=2\overline{N}_{2j}^{+}(n)$ and $\overline{M}_{2j}(n) = 2\overline{M}_{2j}^{+}(n)$ implies that a complete analogue of (\ref{rcforpart}) should hold, namely

\begin{equation} \label{evenm2n2}
\overline{M}_{2j}(n) > \overline{N}_{2j}(n)
\end{equation}

\noindent for all $j$, $n \geq 1$. Motivated by our present work, Jennings-Shaffer \cite{js} has proven (\ref{evenm2n2}) using the Bailey pair techniques from \cite{gar2}. See also \cite{gjs} for the case $k=1$. It would still be interesting to see if the techniques in \cite{ack} can be used to prove (\ref{m2n2}) (and thus (\ref{evenm2n2})) and discover a combinatorial meaning for $\overline{M}_{k}^{+}(n) - \overline{N}_{k}^{+}(n)$. Moreover, there is an inequality of note which has a similar flavor to (\ref{rcforpart}). If we consider the rank moment

\begin{equation*}
\overline{N2}_{k}(n) := \sum_{m \in \mathbb{Z}} m^{k} \overline{N2}(m,n)
\end{equation*}

\noindent where $\overline{N2}(m,n)$ is the number of overpartitions of $n$ with $M_2$-rank $m$ \cite{love2}, then Mao \cite{mao} has proven that

\begin{equation} \label{n2jn22j}
\overline{N}_{2j}(n) > \overline{N2}_{2j}(n)
\end{equation}

\noindent for all $j \geq 1$, $n \geq 2$. Another proof of (\ref{n2jn22j}) using the similarly defined positive rank moment $\overline{N2}_{k}^{+}(n)$ can be found in \cite{lrs}. It is still not known what $\overline{N}_{k}^{+}(n) - \overline{N2}_{k}^{+}(n)$ counts. While proving Corollary \ref{2case} and Theorem \ref{combo}, we observed the following. First, it appears that for all integers $m \geq 3$.

\begin{equation}\label{hqcon}
\frac{1}{(q)_{\infty}} (h(q) - m h(q^{m}))
\end{equation}

\noindent has positive power series coefficients for all positive powers of $q$. Second, numerical computations suggest that
\begin{equation}\label{akbk}
A_{k}(n) \geq B_{k}(n)
\end{equation}
for all $n$, $k \geq 1$. Finally, asymptotic methods reveal that the inequalities \eqref{m2n2} and \eqref{akbk}, and the positivity of the coefficients of \eqref{hqcon} are valid for large enough integers $n$ \cite{kks, rol}. However, it is still desirable to find $q$-theoretic or combinatorial proofs of these result, which shows that these conjectures are true for all positive integers.  We leave these questions to the interested reader.


\section*{Acknowledgements}
The authors thank the anonymous referee for the valuable comments. In particular, the current proof of Lemma \ref{la} is based on the referee's suggestion.

\end{document}